\documentclass[reqno]{amsart} 
\usepackage{amscd,amssymb,latexsym,subfigure,verbatim,hyperref,epsfig,amsmath, xypic,supertabular}
\usepackage{tikz}
\usetikzlibrary{matrix, arrows}
\usepackage{tikz-cd}
\usepackage{enumitem}
\DeclareGraphicsExtensions{.pdf,.jpg, .png}
\usepackage{psfrag}
\usepackage{multicol}

\xyoption{all}
\usepackage{graphicx,amscd,amssymb,latexsym,subfigure,verbatim,hyperref,epsfig,supertabular}
\usepackage[mathscr]{eucal}
\usepackage{mathrsfs}
\usepackage{fullpage}
\usepackage{graphicx}
\usepackage{hyperref}
\usepackage{amssymb}

\newcommand{\R}{{\mathbb R}}

\newcommand{\p}{{\mathbb P}}

\newcommand{\Crk}{\dim C^r_k(\Delta)}
\newcommand{\Cr}{{\mathcal C}^r(\Delta)}

\newcommand{\coker}{\mathop{\rm coker}\nolimits}
\newtheorem{defn0}{Definition}[section]
\newtheorem{prop0}[defn0]{Proposition}
\newtheorem{conj0}[defn0]{Conjecture}
\newtheorem{thm0}[defn0]{Theorem}
\newtheorem{lem0}[defn0]{Lemma}
\newtheorem{corollary0}[defn0]{Corollary}
\newtheorem{example0}[defn0]{Example}
\newtheorem{question0}[defn0]{Question}

\newenvironment{defn}{\begin{defn0}}{\end{defn0}}

\newenvironment{thm}{\begin{thm0}}{\end{thm0}}
\newenvironment{lem}{\begin{lem0}}{\end{lem0}}

\numberwithin{equation}{section}

\begin{document}

\title[A new bound for smooth spline spaces]%
{A new bound for smooth spline spaces}

\author{Hal Schenck}
\thanks{Schenck supported by NSF 1818646}
\address{Schenck: Mathematics
  Department \\ Auburn University\\
  Auburn \\ AL 36849\\ USA}
\email{hks0015@auburn.edu}

\author{Mike Stillman}
\thanks{Stillman and Yuan supported by NSF 1502294}
\address{Stillman: Mathematics Department \\ Cornell University \\
  Ithaca \\ NY 14850\\ USA}
\email{mike@math.cornell.edu}

\author{Beihui Yuan}
\address{Yuan: Mathematics Department \\ Cornell University \\
  Ithaca \\ NY 14850\\ USA}
\email{by238@math.cornell.edu} 

\subjclass[2000]{Primary 41A15, Secondary 13D40, 52C99} \keywords{spline, dimension formula, cohomology}

\begin{abstract}
\noindent 
For a planar simplicial complex $\Delta \subseteq \mathbb{R}^2$,
Schumaker proves in \cite{schu} that a lower bound on the dimension of the space 
$C^r_k(\Delta)$ of planar splines of smoothness $r$ and degree $k$ on $\Delta$ is given by a
polynomial $P_\Delta(r,k)$, and Alfeld-Schumaker show in \cite{as1} that 
$P_\Delta(r,k)$ gives the correct dimension when $k \ge 4r+1$.
Examples due to Morgan-Scott, Tohaneanu, and Yuan 
show that the equality $\Crk =P_\Delta(r,k)$ can fail for $k \in \{2r, 2r+1\}$. In
this note we
prove that the equality $\Crk= P_\Delta(r,k)$ cannot hold in general for $k \le \frac{22r+7}{10}$.
\end{abstract}
\maketitle


\section{Introduction}\label{sec:intro}
Let $\Delta$ be a triangulation of a simply connected 
polygonal domain in $\mathbb{R}^2$ having $f_1$ interior edges and
$f_0$ interior vertices. A landmark result in approximation theory is
the 1979 paper of Schumaker \cite{schu}, showing that for any 
triangulation $\Delta$, any smoothness $r$ and any degree $k$, the 
dimension of the vector space $C^r_k(\Delta)$ of splines of smoothness
$r$ and degree at most $k$ is bounded below by 
\begin{equation}
P_\Delta(r,k) =  {k+2 \choose 2} + {k-r+1 \choose 2}f_1 - \left( {k+2 \choose 2} - {r+2 \choose 2}\right)f_0 + \sigma,
\end{equation}
where $\sigma = \sum\sigma_i, \sigma_i = \sum_j
\max\{(r+1+j(1-n(v_i))), 0 \},\mbox{ and } n(v_i)$ is the number of
distinct slopes at an interior vertex  $v_i.$ In \cite{as1},
Alfeld-Schumaker prove for $k\ge 4r+1$, $\Crk =P_\Delta(r,k)$,
Hong \cite{hong} shows equality holds for $k \ge 3r+2$, and
\cite{as1} shows equality for $k \ge 3r+1$ and generic $\Delta$.

When the degree $k$ is small compared to the order of smoothness,
formula $(1.1)$ can fail to give the correct value for $\Crk$: a 1975 example of Morgan-Scott shows it
fails for $(r,k)=(1,2)$. In \cite{ss} it was conjectured that
$\Crk=P_\Delta(r,k)$ for $k \ge 2r+1$, but a recent example \cite{sy} shows that
equality fails for $(r,k)=(2,5)$. In 1974, Strang \cite{strang} conjectured that for
$(r,k)=(1,3)$ the formula holds for a generic triangulation. 

In \cite{b}, Billera used algebraic methods to prove Strang's
conjecture, winning the Fulkerson prize for his work.  A number of
subsequent papers \cite{br}, \cite{br1}, \cite{ds}, \cite{r1},
\cite{r2}, \cite{ss1}, \cite{ss2}, \cite{st83}, \cite{y} use tools
from algebraic geometry to study
splines. The translation to algebraic geometry takes the set of splines of all polynomial
degrees $k$, and packages it as a vector bundle $\Cr$ on $\p^2$.
The discrepancy between $P_\Delta(r,k)$ and the actual dimension in degree $k$ is then captured by the dimension
$h^1(\Cr(k))$ of the first cohomology of $\Cr$. 

The examples above do not
preclude the possibility that $\Crk=P_\Delta(r,k)$ holds for every
triangulation $\Delta$ if $k \ge 2r+2$. Our main result shows this is impossible:
\begin{thm}There is no constant $c$ so that $\dim C^r_k(\Delta) =
  P_\Delta(r,k)$ for all $\Delta$ and all $k \ge 2r+c$. In particular, there exists a
  planar simplicial complex $\Delta$ for which 
\[
h^1(\Cr(k)) \ne 0 \mbox{ for all }k \le  \frac{22r+7}{10}.
\]
\end{thm}
\noindent This shows there exists a simplicial complex $\Delta$ such that 
$\dim C^r_k(\Delta) > P_\Delta(r,k)$ for all $k \le
\frac{22r+7}{10}$. For formula (1.1) to yield the
correct value for $\Crk$ for every triangulation $\Delta$, we must
have 
\begin{center}$k>\frac{22r+7}{10} >2.2r$. \end{center}
\section{Algebraic preliminaries}
Billera's construction in \cite{b} computes the $C^1$ splines as the
top homology module of a certain chain complex. An introduction to
homology and chain complexes aimed at a general audience appears in \cite{AGAT}, so the presentation below is terse. 
The paper \cite{ss1} introduces a modification of Billera's
construction, allowing a precise splitting of the contributions
to $\Crk$ into parts depending, respectively, on local and global geometry. 
\begin{defn}\label{compC} For a planar simplicial complex $\Delta$,
  let $\Delta_i$ be the set of interior faces of dimension $i$ $($all
  triangles are considered interior$)$. For $\tau \in \Delta_1$, let
  $l_\tau$ be a linear form vanishing on $\tau$, and for $v \in
  \Delta_0$, let $J(v)$ be the ideal generated by
  $l_\tau^{r+1}$, with $\tau$ ranging over all interior edges
  containing $v$. Construct a complex of $R=\R[x_1,x_2,x_3]$ modules
  as below, with differential $\partial_i$ the usual boundary operator in relative $($modulo $\partial(\Delta))$ homology.
\[
{\mathcal R}/{\mathcal J}: \mbox{    }0 \longrightarrow \bigoplus\limits_{\sigma \in \Delta_2} R
\stackrel{\partial_2}{\longrightarrow} \bigoplus\limits_{\tau \in \Delta_1} R/l_{\tau}^{r+1} 
\stackrel{\partial_{1}}{\longrightarrow} \bigoplus\limits_{v \in
  \Delta_0} R/J(v) \longrightarrow 0.
\]
\end{defn}
By construction, $H_2({\mathcal R}/{\mathcal J})$ is a graded
$R$-module, consisting of the set of splines of all degrees, and defines the sheaf
$\Cr$. It is easy to show that 
\[
H_0({\mathcal R}/{\mathcal J})=0 \mbox{ and } H_1({\mathcal R}/{\mathcal J}) = \bigoplus\limits_{k \ge 0}H^1(\Cr(k)).
\]
In particular, $\Crk = P_\Delta(r,k)+\dim_{\R}H^1(\Cr(k))$. Recall
that a {\em syzygy} on an ideal $\langle f_1,\ldots, f_k\rangle$ is a
polynomial relation on the $f_i$. For an interior vertex $v$, $J(v) = \langle
l_{\tau_1}^{r+1},\ldots, l_{\tau_n}^{r+1}\rangle$, so a syzygy on
$J(v)$ is of the form $\sum_{i=1}^n s_i \cdot l_{\tau_i}^{r+1} =0.$  A
main result of \cite{ss1} is 
\begin{thm}\label{presH1}
The module $H_1({\mathcal R}/{\mathcal J})$ is given by generators and
relations as 
\[
H_1({\mathcal R}/{\mathcal J}) \simeq \Big(\bigoplus\limits_{\tau \in \Delta_1^o} R(-r-1)\Big)/S, \mbox{ where }
\]
\begin{itemize}
\item The set $\Delta_1^o$ consists of totally interior edges $\tau$:
  neither vertex of $\tau$ is in $\partial(\Delta)$.
\item $S = \bigoplus\limits_{v\in \Delta_0} Syz(v)$: the direct sum
  of the syzygies on $J(v)$ at each interior vertex. 
\end{itemize}
\end{thm}
\noindent Hence $H_1({\mathcal R}/{\mathcal J})$ is the quotient of a free
module with a generator for each totally interior edge $\tau$ by vectors 
of polynomials of the form $(s_1,\ldots,s_n)$. Note that if two
totally interior edges $\tau_1, \tau_2$ with the same slope meet at a vertex, then there
is a degree zero syzygy between them, and $S$ will have a column with nonzero constant entries.
\section{Proof of Theorem}
\noindent Following \cite{sy}, we consider the simplicial complex
$\Delta$ below.
\begin{center}
\includegraphics[trim = 45mm 90mm 45mm 90mm, clip,width=55mm,height=55mm]{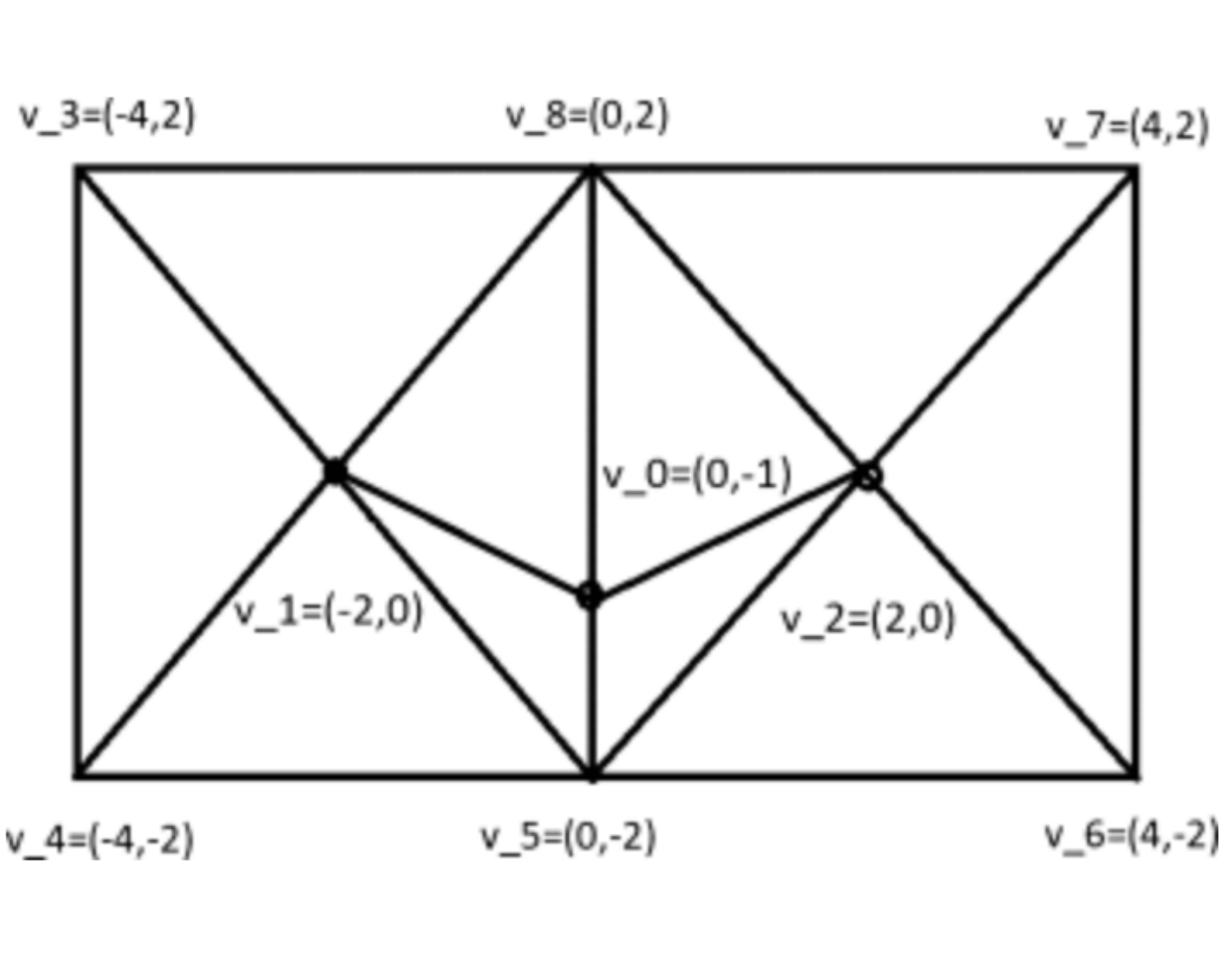}
\end{center}
By Theorem~\ref{presH1}, the discrepancy module $H_1({\mathcal R}/{\mathcal J})$ has two generators.
There are three interior vertices, and we need to quotient by the
syzygies at each vertex. Note that each vertex has only three edges
with distinct slopes attached, hence we must compute the syzygies on
ideals of the form 
\[
\langle l_1^{r+1},l_2^{r+1},l_3^{r+1} \rangle
\]
The key is that this is a local question, so after translating a
vertex so it lies at the origin, we have an ideal in two variables
(recall that because we homogenized the problem, our points now lie
in $\p^2$, so the linear forms defining edges are homogeneous in three
variables). The paper \cite{gs} gives a precise description of the
syzygies on any ideal generated by powers of linear forms in two
variables. In the case of three forms as above there are only two syzygies, in
degrees
\[
\Big\lfloor \frac{r+1}{2} \Big\rfloor \mbox{ and } \Big\lceil \frac{r+1}{2} \Big\rceil
\]
Specializing to the case where $r+1 = 4j$, we see that there are two
syzygies, both of degree $2j$. Next, we note that two of the three
vertices are connected to one totally interior edge and two edges
which touch the boundary, so writing the six relations (two syzygies
on each of the three interior vertices) as a matrix, we see that 
$H_1({\mathcal R}/{\mathcal J}) \simeq R^2(-r-1)/S$, where
\[
S = \;\;{\small \left[ \begin{array}{*{6}c}
s_{11} & s_{12} & s_{13} & s_{14} & 0 & 0 \\
0 & 0& s_{23} & s_{24} & s_{25} & s_{26}
\end{array} \right]}
\]
As noted above, the rows correspond to the generators for
$H_1({\mathcal R}/{\mathcal J})$:  the first row corresponds to the
totally interior edge $\overline{v_0v_1}$ and the second row to the
totally interior edge $\overline{v_0v_2}$; let $l_{ij}$ denote a
nonzero linear form vanishing on $\overline{v_iv_j}$.  

The first two
columns of $S$ correspond to the two syzygies at vertex $v_1$, the
second two columns to the syzygies at vertex $v_0$, and the last two
columns to the syzygies at vertex $v_2$. Since the syzygies
at $v_0$ are on the ideal
\[
\langle l_{01}^{r+1}, l_{02}^{r+1},l_{08}^{r+1}\rangle, 
\]
the third and fourth columns of $S$ have no zero
entries, because the syzygies involve both generating edges
$\overline{v_0v_1}, \overline{v_0v_2}$. In contrast, the syzygies at
$v_1$ are on the ideal
\[
\langle l_{01}^{r+1},l_{13}^{r+1},l_{14}^{r+1}\rangle.
\]
Hence in the matrix $S$, only the component of the syzygy involving $l_{01}^{r+1}$ appears--there is no part of the syzygy involving $l_{02}^{r+1}$.
This also explains why the rightmost two columns of $S$ have nonzero entry only in the second row. For the next lemma, we need some concepts from commutative algebra.
\begin{defn}
An ideal $I = \langle f_1, \ldots, f_k\rangle \subseteq R$ with $k$
minimal generators is a {\em complete intersection} if each $f_i$ is not a zero divisor on $R/\langle f_1,
\ldots, f_{i-1}\rangle$. Equivalently, the map 
\[
R/\langle f_1,\ldots, f_{i-1}\rangle \stackrel{\cdot f_i}{\longrightarrow} R/\langle f_1,\ldots, f_{i-1}\rangle
\]
is an inclusion. 
\end{defn}

From a geometric standpoint, being a complete intersection means that the locus
$V(f_1,\ldots,f_k)$ where the $f_j$ simultaneously vanish has
codimension equal to $k$. In particular, an ideal $I$ minimally generated
by $k$ elements is a compete intersection if it has codimension $k$,
and an {\em almost complete intersection} if it has codimension
$k-1$. 
 
 \begin{defn} Let $I, J$ be ideals in a ring $R$. Then the {\em colon
     ideal} 
   \[
     I:J = \{f \in R \mid f\cdot j \in I \mbox{ for all } j \in J\}
   \]
 \end{defn}
There is a nice connection of colon ideals to syzygies: if $I=\langle
f_1, \ldots, f_k\rangle$ and
\[
  \sum\limits_{i=1}^k a_if_i=0
\]
is a syzygy on $I$, then $a_k \in \langle f_1, \ldots, f_{k-1}\rangle
: \langle f_k \rangle$. We shall make use of this in the next lemma.

\pagebreak

\begin{lem}
The ideals $I_1 = \langle s_{11}, s_{12}\rangle$ and $I_2 = \langle
s_{25}, s_{26}\rangle$ are complete intersections.
\end{lem}
\begin{proof} An ideal with two generators $f,g$ is a complete
  intersection when $f$ and $g$ are relatively prime, or
  equivalently when the unique minimal syzygy on $f,g$ is given by
  $f\cdot g - g\cdot f = 0$. The ideal $\langle l_1^{r+1},l_2^{r+1},l_3^{r+1}
\rangle$ is an almost complete intersection, which 
means that two generators, say $\{ l_1^{r+1},l_2^{r+1} \}$ are a
complete intersection. Proposition 5.2 in \cite{be} proves an almost complete
intersection is directly linked to a Gorenstein ideal. In this case
the linked ideal is $\langle  l_1^{r+1},l_2^{r+1} \rangle : l_3^{r+1} =
\langle s_{11}, s_{12}\rangle$. A homogeneous Gorenstein ideal in two variables is a complete
intersection, so the result follows.
\end{proof}
\noindent We're now ready to put the pieces together. Define 
\[
\phi = \;\;{\small \left[ \begin{array}{*{2}c}
s_{13} & s_{14}  \\
s_{23} & s_{24} 
\end{array} \right].}
\]
Then $H_1({\mathcal R}/{\mathcal J})$ may be presented as the cokernel
of the map
\[
R^2(-6j) \stackrel{\phi}{\longrightarrow} R(-4j)/I_1 \bigoplus R(-4j)/I_2
\] 
The Hilbert function of a graded module $M$ takes as input an integer
$t$, and gives as output the dimension of the vector space
$M_t$. Since $I_i$ is a complete intersection with two generators in degree $2j$, there are exact sequences:
\[
0 \longrightarrow R(-4j) \longrightarrow R(-2j)^2 \longrightarrow R \longrightarrow R/I_i \longrightarrow 0.
\]
Tensoring this exact sequence with $R(-4j)$ yields a sequence whose rightmost term is a direct summand of the target of the map $\phi$.
When $k \ge 2r+2 = 8j$ (so that all the modules in the exact sequence above contribute), taking the Euler characteristic of the sequence 
and using that $HF(R(-i),k)={k-i+2 \choose 2}$ yields
\[
\begin{array}{ccc}
HF(R^2(-6j),k) &=& (k-6j+2)(k-6j+1)\\
HF(R(-4j)/I_1 \bigoplus R(-4j)/I_2, k) & = &
                                             (k-4j+2)(k-4j+1)\\
                                           && -2(k-6j+2)(k-6j+1)\\
                                           &&+(k-8j+2)(k-8j+1)
\end{array}
\]
Therefore the Hilbert function of the target of $\phi$ minus
the Hilbert function of the source of $\phi$ is 
\[
-k^2 + (12j-3)k -28j^2+18j-2, 
\]
which has two real roots, the larger at 
\[
k= 6j-3/2 + \frac{\sqrt{32j^2+1}}{2} >(6+2\sqrt{2})j -3/2 > 8.8j-2.2+.7
= \frac{22r+7}{10}.
\]
We have been working with the assumption that $r+1 = 4j$; for $r \ge 7$ the condition $k \ge 8j$ holds and we've shown the cokernel of $\phi$ must be nonzero in degree  $\le \frac{22r+7}{10}.$ For $r=3$ and $j=1$ the condition that $k \ge 8$ fails--the larger root is at approximately $7.4$. In this case, a direct computation verifies that $\coker(\phi)$ is nonzero in degree $7$. The same line of argument works with a minor modification for $(r+1 \mod 4) \in \{1,2,3\}$, with no change in the bound, and concludes the proof. $\qed$
\pagebreak

\noindent{\bf Remarks and Open Questions}: 
The triangulation $\Delta$ appearing in $\S 3$ is the only known
triangulation for which $P_\Delta(r,2r+1) \ne \dim
C^r_{2r+1}(\Delta)$. For $r \le 70$, computations show that the maximal value for
which $H_1({\mathcal R}/{\mathcal J_\Delta}) \ne 0$ is $\Big\lfloor
\frac{9r+2}{4}\Big\rfloor =\Big\lfloor \frac{45r+10}{20}\Big\rfloor
\ge \Big\lfloor \frac{44r+14}{20}\Big\rfloor =\Big\lfloor
\frac{22r+7}{10} \Big\rfloor$. In particular the bound of Theorem 1.1
is quite close to optimal for $\Delta$. This raises two interesting
questions. 
\vskip .05in
\begin{enumerate}
\item Is it possible to
lower the value of $k$ such that $\Crk =P_\Delta(r,k)$ holds
for all $\Delta$? 
\item Is it possible to raise  the value of $k$ such that $\Crk > P_\Delta(r,k)$ holds for some $\Delta$? 
\end{enumerate}
%
\renewcommand{\baselinestretch}{1.0}
\small\normalsize 

\bibliographystyle{amsalpha}

\end{document}